\documentclass{amsart}
\usepackage[utf8]{inputenc}
\usepackage{amsmath,eucal,graphicx,amsxtra,hyperref}
\usepackage{amsrefs}

\input xy
\xyoption{all}

\newtheorem{theorem}{Theorem}[section]

\theoremstyle{definition}
\newtheorem{definition}[theorem]{Definition}

\theoremstyle{remark}

\newcommand{\Cyc}{{\mathcal C}}
\newcommand{\ZZ}{{\mathcal Z}}

\newcommand{\C}{\mathbb{C}}

\newcommand{\Z}{\mathbb{Z}}

\renewcommand{\P}{\mathbb{P}}
\newcommand{\PP}{\mathbb{P}}
\newcommand{\GG}{\mathbb{G}}

\def\G{{\mathcal G}}
\def\GG{{\mathcal G}}

\begin{document}
\title{The Incidence Correspondence and its associated maps in Homotopy}
\date{\today}
\author[L. \,E. Lopez]{Luis E. Lopez}
\address{Max-Planck-Institut f\"ur Mathematik\\
         Vivatsgasse 7\\
	 D-53111 Bonn \\
	 Germany
}
\curraddr{Max-Planck-Institut f\"ur Mathematik\\
         Vivatsgasse 7\\
	 D-53111 Bonn \\
	 Germany
}
\email{llopez@mpim-bonn.mpg.de}

\begin{abstract}
The incidence correspondence in the grassmannian which determines the tautological bundle 
defines a map between cycle spaces on grassmannians. These cycle spaces decompose canonically
into a product of Eilenberg-MacLane spaces. These decompositions and the associated maps are calculated up to homotopy.
\end{abstract}

\maketitle

\section{Introduction}

Let $\Cyc^p_d(\PP^n)$ denote the space of algebraic cycles of codimension $p$ and degree $d$ in $\PP^n$. This set can be given the
structure of an algebraic variety via the Cayley-Chow-Van der Waerden embedding which takes an irreducible cycle $X$ 
into $\Psi(X)$ where the Chow Form $\Psi$ is obtained from the following incidence correspondence:
$$\xymatrix{ & \{ (x,\Lambda) \in \PP^n \times \GG_{p-1}(\PP^n) \mid x \in \Lambda \} \ar[dl]_{\pi_1} \ar[dr]^{\pi_2} \\
\PP^n & &  \GG_{p-1}(\PP^n) \\
 & \Psi(X) = \pi_2\pi_1^{-1}(X)&
} 
$$
This map is then extended additively to the topological monoid $\Cyc^p(\PP^n)$ of all cycles and to its naive group completion $\ZZ^p(\PP^n)$. The cycle $\Psi(X)$ has 
codimension $1$ in the grassmannian, therefore $\Psi$ defines a map
$$
\Psi: \ZZ^p(\PP^n) \rightarrow \ZZ^1\left(\GG_{p-1}(\PP^n)\right)
$$
Moreover, if $B = \bigoplus B_d$ denotes the coordinate ring of the grassmannian in the Pl\"ucker embedding, then every irreducible hypersurface $Z$ of degree $d$ in the
grassmannian is given by an element $f \in B_d$ defined uniquely up to a constant factor (see \cite{GKZ}), this defines a grading on the space of hypersurfaces
in the grassmannian and with respect to this grading the map above preserves the degree.\\

The topology thus inherited defines in turn a topology in the space $\Cyc^p(\PP^n)$ of all codimension $p$ algebraic cycles in $\PP^n$ 
(c.f. \cite {FM} for this and other equivalent definitions)  The results presented here are of the following type:

\begin{theorem}
 The Chow Form map 
$$\Psi: \ZZ^p(\PP^n) \rightarrow \ZZ^1\left(\GG_{p-1}(\PP^n)\right)$$
can be represented with respect to canonical decompositions of the corresponding spaces in the following way:
$$
\xymatrix{ \prod_{j=0}^p K(\Z,2j) \ar[r]^-{p} & K(\Z,0) \times K(\Z,2)}
$$
where $p$ is the projection into the first two factors of the product
\end{theorem}
hence, the Chow Form map can be interpreted as the classifying map of a (non trivial) line bundle in the space of 
algebraic cycles in $\P^n$, 
the class of this line bundle generates
$H^2\left(\ZZ^1\left(\GG_{p-1}(\PP^n)\right)\right)$. \\

An equivalent way of looking at the theorem is via the chain of inclusions 
$$
\GG^p(\PP^n) \subset \ZZ^p(\PP^n) \subset \ZZ^1\left(\GG_{p-1}(\PP^n)\right)
$$

The first space classifies $p$-dimensional bundles (not all of them), the second space classifies all total integer cohomology classes
in $\bigoplus_{j=0}^p H^{2j}(-;\Z)$ and the third space classifies all  integer cohomology classes in $H^0(-;\Z) \times H^{2}(-;\Z)$. The corresponding
maps associated to the inclusions are the total chern class map (see \cite{LM}) and the projection into the first two factors. \\ 


\section{The Chow Form Map}
In this section we prove the theorem regarding the chow form map using an explicit description of the map. In order to
state the results and its proofs we recall some definitions and facts about Lawson Homology. A survey of these and
other related results is given in \cite{lawsurvey}.

\begin{definition}
Let $X$ be a projective variety. The {\em Lawson Homology  groups} $L_pH_n(X)$ of $X$ are defined by
$$
L_pH_n(X) := \pi_{n-2p}\left(\ZZ_p(X)\right)
$$
where $\ZZ_p(X)$ is the naive group completion of the Chow monoid $\Cyc_p(X)$ of all $p$-dimensional
effective algebraic cycles in $X$
\end{definition}

These groups stand between the group of algebraic cycles modulo algebraic equivalence ${\mathcal A}_p(X) = L_pH_{2p}(X)$ and the
singular homology group $H_{n}(X) = L_0H_n(X)$. Friedlander and Mazur defined a {\em cycle map} between the Lawson Homology groups and
the singular homology groups
$$
s^{(p)}_X:L_pH_n(X) \rightarrow H_n(X;\Z)
$$
this map will be referred to as the F-M map. \\

The following results of Lima-Filho \cite{L-FCyclemap} will be used throughout the section.

\begin{theorem}[Lima-Filho, \cite{L-FCyclemap}]
\label{smaps}
The F-M cycle map coincides with the composition
$$
\xymatrix{
L_pH_n(X) = \pi_{n-2p}(\ZZ_{n-p}(X)) \ar[r]^-{e_{*}} & \pi_{n-2p}({\mathfrak Z}_{2p}(X)) \ar[r]^-{{\mathcal A}} & H_n(X; \Z)
}
$$
where $e_{*}$ is the map induced by the inclusion
$$e: \ZZ_m(X) \rightarrow \mathfrak{Z}_{m}(X)$$ 
of the space of algebraic cycles into the space of all  integral currents and ${\mathcal A}$ is the Almgren isomorphism
defined in \cite{Almgren}. In particular, the F-M is functorial and is compatible with proper push-forwards and flat-pullbacks of cycles.
\end{theorem}

\begin{theorem}[Lima-Filho, \cite{L-FCyclemap}]
\label{inclusion}
If $X$ is a projective variety with a cellular decomposition in the sense of Fulton, (i.e., $X$ is an algebraic
cellular extension of $\emptyset$) then the inclusion
$$\ZZ_p(X) \hookrightarrow \mathfrak{Z}_{2p}(X)$$
into the space $\mathfrak{Z}_{2p}(X)$ of integral currents is a homotopy equivalence
\end{theorem}

We also recall some facts about the cohomology of the Grassmannian. We follow the notation of \cite{Fulton} Chapter $14$.

\begin{definition}
\label{schubert}
Let $\GG_k(\P^n)$ be the Grassmann variety of $k$-dimensional linear spaces in $\P^n$. $\GG_k(\P^n)$ is a smooth algebraic
variety of complex dimension $d :=(k+1)(n-k)$.
The {\em special Schubert classes} are the homology classes $\sigma_m \in H_{2(d-m)}(\GG_k(\P^n))$ defined by the cycle class of
$$
\sigma_m := \{ L \in \GG_k(\P^n) \mid L \cap A \neq \emptyset \}
$$
where $A$ is any linear subspace of $\P^n$ of codimension $k+m$.
\end{definition}

\begin{theorem}
The integral cohomology of the grassmannian $\GG_k(\P^n)$ is  generated by the Poincare duals $c_m(Q)$ of the
special Schubert classes $\sigma_m$. These Poincare duals are the chern classes of the universal quotient bundle $Q$.
\end{theorem}

\begin{theorem} \label{chowform}
 The Chow Form map 
$$\Psi: \ZZ^p(\PP^n) \rightarrow \ZZ^1\left(\GG_{p-1}(\PP^n)\right)$$
can be represented with respect to canonical decompositions of the corresponding spaces in the following way:
$$
\xymatrix{ \prod_{j=0}^p K(\Z,2j) \ar[r]^-{p} & K(\Z,0) \times K(\Z,2)}
$$
where $p$ is the projection into the first two factors of the product
\end{theorem}
\begin{proof}
The homotopy equivalences
$$\ZZ^p(\PP^n) \simeq \prod_{j=1}^p K(\Z,2j)$$
and 
$$
\ZZ^1\left(\GG_{p-1}(\PP^n)\right) \simeq K(\Z,0) \times K(\Z,2)
$$
are a consequence of theorem \ref{inclusion} and Almgren's theorem which asserts that
$$\pi_i(\mathfrak{Z}_k{X}) \cong H_{i+k}(X)$$
for the second homotopy equivalence, if $d=\dim_{\C} \GG_{p-1}(\PP^n)$ then 
\begin{equation} 
\pi_i\left(\ZZ_{d-1}\GG_{p-1}(\PP^n)\right) \cong 
\begin{cases}
H_{2d-1}\left(\GG_{p-1}(\PP^n)\right) \cong \Z & \text{if $i=0$}\\
H_{2d}\left(\GG_{p-1}(\PP^n)\right) \cong \Z& \text{if $i=2$}\\
H_{d-1+i}\left(\GG_{p-1}(\PP^n)\right) = 0 & \text{otherwise}
\end{cases}
\end{equation}
this last calculation being well known for grassmannians (see \cite{BottandTu}). Now we
use theorem \ref{smaps} to calculate the induced maps in homotopy. We have the following
commutative diagram:
$$
\xymatrix{
\pi_{m}(\ZZ_{n-p}(\PP^n))=L_{n-p}H_{m+2(n-p)}(\P^n) \ar[r]^-{\Psi} \ar[d]_-s &  \pi_{m}(\ZZ_{d-1}\GG_{p-1}(\PP^n)) \cong L_{d-1}H_{m+2(g-1)}(\GG_{p-1}(\PP^n)) \ar[d]^-s \\
H_{m+2(n-p)}(\P^n) \ar[r]^-{\tilde{\Psi}} & H_{m+2(d-1)}(\GG_{p-1}(\PP^n))
}
$$
where $\Psi$ is the Chow Form map and $\tilde{\Psi}$ is the corresponding map in the space of integral currents.

Theorem \ref{inclusion} implies that the vertical maps are isomorphisms.
Since we know the homology of the grassmannian, the only non-zero dimensions in the lower right correspond to the cases
$m=0$ and $m=2$. We will describe explicitly the morphism in these two cases.
\begin{itemize}
 \item[$m=0$] In this case $H_{2(n-p)}(\P^n)$ is generated by the class of an $(n-p)$-plane $\Lambda$ in $\P^n$. The cycle $\tilde{\Psi}(\Lambda)$ is then
$$\tilde{\Psi}(\Lambda)=\{ P \in \GG_{p-1}(\PP^n) \mid \Lambda \cap P \neq \emptyset \}$$
this is precisely definition \ref{schubert} of the special Schubert cycle $\sigma_1$.
 \item[$m=2$] In this case $H_{2 + 2(n-p)}(\P^n)$ is generated by the class of an $(n-(p-1))$-plane $\Lambda$ in $\P^n$. The cycle $\tilde{\Psi}(\Lambda)$ is then
$$\tilde{\Psi}(\Lambda)=\{ P \in \GG_{p-1}(\PP^n) \mid \Lambda \cap P \neq \emptyset \}$$
but by a dimension count this is the whole variety $\GG_{p-1}(\PP^n)$ again, this corresponds to  the special Schubert cycle $\sigma_0$

\end{itemize}

\end{proof}

The results of Lima-Filho about the F-M cycle class map may also be used to prove the following theorem of Lawson and Michelsohn which appeared in
\cite{LM}.

\begin{theorem}
 The complex join pairing in the cycle spaces 
$$\#:\ZZ^q(\P^n) \wedge \ZZ^{q'}(\P^m) \rightarrow \ZZ^{q+q'}(\P^{n+m+1})$$
represents the cup product pairing in the canonical decompositions
$$\cup: \prod_{s=0}^{q}K(\Z,2s) \wedge \prod_{t=0}^{q'}K(\Z,2t) \rightarrow  \prod_{r=0}^{q+q'}K(\Z,2r)$$
i.e. if $i_{2c}$ represents the generator of $H^{2c}(K(\Z,2c);\Z) \cong \Z$ then
$$(\#)^{*}(i_{2c}) = \sum_{\begin{aligned} a+b = c \\ a > 0 \\ b > 0 \end{aligned}}i_{2a} \otimes i_{2b}$$
\end{theorem}
\begin{proof} The linear join of two irreducible varieties $X \subset \P^n$ and $Y \subset \P^m$ with defining
ideals $\langle F_i \rangle \in \C[x_0,\ldots,x_n]$ and $\langle \G_j \rangle \in \C[y_0,\ldots,y_m]$ is the variety in $\P^{n+m+1}$ defined
by the ideal $\langle F_i,G_j \rangle \in \C[x_0,\ldots,x_n,y_0,\ldots,y_m]$. A synthetic construction of this pairing is given in
\cite{LM}. The join pairing $\#$ is obtained by extending bi-additively to all cycles and taking the induced pairing
in the smash product. The join pairing induces a bilinear pairing
$$ \pi_s(\ZZ^q(\P^n)) \times  \pi_r(\ZZ^{q'}(\P^m)) \rightarrow \pi_{r+s}(\ZZ^{q+q'}(\P^{n+m+1}))$$
which corresponds to a homomorphism
$$ \#_*:L_{n-q}H_{s+2(n-q)}(\P^n) \otimes L_{m-q'}H_{r+2(m-q')}(\P^m) \rightarrow L_{(n-q)+(m-q')+1}H_{s+2(n-q)+r+2(m-q')+2}(\P^{n+m+1})$$
This homomorphism is non-zero only when $s$ and $r$ are even, so we may assume $s=2a$ and $r= 2b$. Now we take the F-M map and we get the 
following commutative diagram
$$\xymatrix{
L_{n-q}H_{s+2(n-q)}(\P^n) \otimes L_{m-q'}H_{r+2(m-q')}(\P^m) \ar[r]^-{\#_*} \ar[d] &  L_{(n-q)+(m-q')+1}H_{s+2(n-q)+r+2(m-q')+2}(\P^{n+m+1}) \ar[d] \\
H_{s+2(n-q)}(\P^n) \otimes H_{r+2(m-q')}(\P^m) \ar[r]^-{\phi} & H_{s+2(n-q)+r+2(m-q')+2}(\P^{n+m+1})
}
$$
where we define $\phi$ on the generators and we extend it bilinearly, namely, if the cycle classes of the planes $\Lambda_1 \in H_{2a+2(n-q)}(\P^n)$ and
$\Lambda_2 \in H_{2b+2(m-q')}(\P^m)$ are generators of the corresponding groups, then $\phi([\Lambda_1] \otimes [\Lambda_2])=[\Lambda_1 \# \Lambda_2]$, this last class
is again a generator of $H_{2(a+b)+2(n-q)+2(m-q')+2}(\P^{n+m+1})$. The vertical F-M maps are isomorphisms by theorem \ref{inclusion}, so we get that the generator $i_{2c}$
gets pulled back precisely to the sum of the generators $i_{2a} \otimes i_{2b}$ with $a+b=c$.

\end{proof}

\section{Generalizations}

The proof of theorem \ref{chowform} suggests a generalization of the result. Instead of taking the Chow form map we will look at general correspondences.

\begin{definition}
Let $\Sigma_k$ denote the {\em incidence correspondence} defined by
$$\xymatrix{ & \Sigma_k := \{ (x,\Lambda) \in \PP^n \times \GG_{k}(\PP^n) \mid x \in \Lambda \} \ar[dl]_{\pi_1} \ar[dr]^{\pi_2} \\
\PP^n & &  \GG_{k}(\PP^n) \\
 & \Psi_k(X) = \pi_2\pi_1^{-1}(X)&
} 
$$
\end{definition}

The map $\pi_1$ is flat and the map $\pi_2$ is proper (see \cite{harris}). Therefore $\Psi_k$ is a well defined map of cycle spaces
$$\Psi_k: \Cyc^p(\P^n) \rightarrow \Cyc^{p-k}(\GG_k(\P^n)).$$

Notice that the incidence correspondence which defines the universal quotient bundle is $\Sigma_{p-1}$. \\

With this notation we have the following

\begin{theorem}
 The map
$$\Psi_k: \ZZ^p(\PP^n) \rightarrow \ZZ^{p-k}\left(\GG_{k}(\PP^n)\right)$$
can be represented with respect to canonical decompositions of the corresponding spaces in the following way
$$
\xymatrix{\prod_{j=0}^p K(\Z,2j) \ar[r]^-{p} & \prod_{r=0}^{p-k} \prod_{\alpha \in H_{2r}(\GG_k(\P^n))}K(\Z,2r)_{\alpha}}
$$
where $p$ is the projection into the factors corresponding to the classes $\sigma_0,\sigma_1,\ldots,\sigma_{(p-k)}$
\end{theorem}
\begin{proof}
The argument follows  the proof of theorem \ref{chowform}. We take the homotopy groups and then we use
the cycle map. In this case we have the following commutative diagram:
$$
\xymatrix{
\pi_{m}(\ZZ_{n-p}(\PP^n))=L_{n-p}H_{m+2(n-p)}(\P^n) \ar[r]^-{\Psi_k} \ar[d]_-s &  \pi_{m}(\ZZ_{m}\GG_{k}(\PP^n)) \cong L_{d-(p-k)}H_{d+2[d-(p-k)]}(\GG_{k}(\PP^n)) \ar[d]^-s \\
H_{m+2(n-p)}(\P^n) \ar[r]^-{\tilde{\Psi_k}} & H_{m+2[d-(p-k)]}(\GG_{k}(\PP^n))
}
$$

Since the homology of the grassmannian is zero in odd degrees we are only concerned with the case $m=2r$. The lower right corner of the diagram
imposes the condition $0 \leq 2r \leq 2(p-k)$.  Once again, the definition of the special Schubert cycles implies that if $\Lambda_{2r + 2(n-p)}$ is
a plane of dimension $2r + 2(n-p)$ (i.e. the class of a generator of $H_{2r + 2(n-p)}(\P^n)$) then $\tilde{\Psi_k}(\Lambda_{2r + 2(n-p)}) = \sigma_{p-(k+r)}$
\end{proof}


\begin{bibsection}
\begin{biblist}
\bib{Almgren}{article}{
   author={Almgren, Frederick Justin, Jr.},
   title={The homotopy groups of the integral cycle groups},
   journal={Topology},
   volume={1},
   date={1962},
   pages={257--299},
   issn={0040-9383},
   review={\MR{0146835 (26 \#4355)}},
}

\bib{BottandTu}{book}{
   author={Bott, Raoul},
   author={Tu, Loring W.},
   title={Differential forms in algebraic topology},
   series={Graduate Texts in Mathematics},
   volume={82},
   publisher={Springer-Verlag},
   place={New York},
   date={1982},
   pages={xiv+331},
   isbn={0-387-90613-4},
   review={\MR{658304 (83i:57016)}},
}

\bib{FM}{article}{
   author={Friedlander, Eric M.},
   author={Mazur, Barry},
   title={Filtrations on the homology of algebraic varieties},
   note={With an appendix by Daniel Quillen},
   journal={Mem. Amer. Math. Soc.},
   volume={110},
   date={1994},
   number={529},
   pages={x+110},
   issn={0065-9266},
   review={\MR{1211371 (95a:14023)}},
}

\bib{Fulton}{book}{
   author={Fulton, William},
   title={Intersection theory},
   series={Ergebnisse der Mathematik und ihrer Grenzgebiete. 3. Folge. A
   Series of Modern Surveys in Mathematics [Results in Mathematics and
   Related Areas. 3rd Series. A Series of Modern Surveys in Mathematics]},
   volume={2},
   edition={2},
   publisher={Springer-Verlag},
   place={Berlin},
   date={1998},
   pages={xiv+470},
   isbn={3-540-62046-X},
   isbn={0-387-98549-2},
   review={\MR{1644323 (99d:14003)}},
}

\bib{GKZ}{book}{
   author={Gel{\cprime}fand, I. M.},
   author={Kapranov, M. M.},
   author={Zelevinsky, A. V.},
   title={Discriminants, resultants, and multidimensional determinants},
   series={Mathematics: Theory \& Applications},
   publisher={Birkh\"auser Boston Inc.},
   place={Boston, MA},
   date={1994},
   pages={x+523},
   isbn={0-8176-3660-9},
   review={\MR{1264417 (95e:14045)}},
}

\bib{lawsurvey}{article}{
   author={Lawson, H. Blaine, Jr.},
   title={Spaces of algebraic cycles},
   conference={
      title={Surveys in differential geometry, Vol.\ II},
      address={Cambridge, MA},
      date={1993},
   },
   book={
      publisher={Int. Press, Cambridge, MA},
   },
   date={1995},
   pages={137--213},
   review={\MR{1375256 (97m:14006)}},
}

\bib{L-FCyclemap}{article}{
   author={Lima-Filho, Paulo},
   title={On the generalized cycle map},
   journal={J. Differential Geom.},
   volume={38},
   date={1993},
   number={1},
   pages={105--129},
   issn={0022-040X},
   review={\MR{1231703 (94i:14027)}},
}

\bib{harris}{book}{
   author={Harris, Joe},
   title={Algebraic geometry},
   series={Graduate Texts in Mathematics},
   volume={133},
   note={A first course;
   Corrected reprint of the 1992 original},
   publisher={Springer-Verlag},
   place={New York},
   date={1995},
   pages={xx+328},
   isbn={0-387-97716-3},
   review={\MR{1416564 (97e:14001)}},
}
\bib{LM}{article}{
   author={Lawson, H. Blaine, Jr.},
   author={Michelsohn, Marie-Louise},
   title={Algebraic cycles, Bott periodicity, and the Chern characteristic
   map},
   conference={
      title={The mathematical heritage of Hermann Weyl},
      address={Durham, NC},
      date={1987},
   },
   book={
      series={Proc. Sympos. Pure Math.},
      volume={48},
      publisher={Amer. Math. Soc.},
      place={Providence, RI},
   },
   date={1988},
   pages={241--263},
   review={\MR{974339 (90d:14010)}},
}


\bib{FL}{article}{
   author={Friedlander, Eric M.},
   author={Lawson, H. Blaine, Jr.},
   title={A theory of algebraic cocycles},
   journal={Ann. of Math. (2)},
   volume={136},
   date={1992},
   number={2},
   pages={361--428},
   issn={0003-486X},
   review={\MR{1185123 (93g:14013)}},
}

\end{biblist}
\end{bibsection}
\end{document}